\documentclass[12pt, reqno]{amsart}

\usepackage{amsmath, amsthm, amssymb}
\usepackage{enumerate}
\usepackage[margin=1.0in]{geometry}
\usepackage{xcolor}
\definecolor{cite}{rgb}{0.30,0.60,1.00}
\definecolor{url}{rgb}{0.00,0.00,0.80}
\definecolor{link}{rgb}{0.40,0.10,0.20}
\usepackage[pdfusetitle,colorlinks,linkcolor=link,urlcolor=url,citecolor=cite,pagebackref,breaklinks]{hyperref}
\usepackage{graphicx}
\usepackage{cleveref}
\usepackage{mathdots}
\usepackage{comment}
\usepackage{mathtools}
\usepackage{physics}

\newtheorem{theorem}{Theorem}[section]

\newtheorem{proposition}[theorem]{Proposition}

\newtheorem{corollary}[theorem]{Corollary}

\theoremstyle{definition}

\theoremstyle{definition}
\newtheorem{remark}[theorem]{Remark}
\theoremstyle{definition}

\newcommand{\zIntegers}{\mathbb{Z}}

\newcommand{\cComplex}{\mathbb{C}}
\newcommand{\multiplicativegroup}[1]{#1^{\times}}
\newcommand{\RealPart}{\mathrm{Re}}
\newcommand{\Hom}{\mathrm{Hom}}

\newcommand{\sizeof}[1]{\left|#1\right|}

\newcommand{\standardForm}[2]{\left\langle #1,#2\right\rangle}

\newcommand{\finiteField}{\mathbb{F}}
\newcommand{\fieldCharacter}{\psi}
\newcommand{\centralCharacter}[1]{\omega_{#1}}
\newcommand{\Ind}[3]{\mathrm{Ind}_{#1}^{#2}\left(#3\right)}

\newcommand{\Whittaker}{\mathcal{W}}
\newcommand{\Contragradient}[1]{#1^{\vee}}
\newcommand{\SpehRepresentation}[2]{\Delta\left(#1, #2\right)}

\newcommand{\IdentityMatrix}[1]{I_{#1}}
\newcommand{\diag}{\mathrm{diag}}

\newcommand{\GL}{\mathrm{GL}}
\newcommand{\UnipotentSubgroup}{U}
\newcommand{\ParabolicForWss}[2]{P_{\qty(#2^{#1})}}
\newcommand{\UnipotentRadicalForWss}[2]{N_{\qty(#2^{#1})}}
\newcommand{\UnipotentRadicalForWssRecursion}[2]{\mathcal{Y}_{c,k}}

\newcommand{\squareMatrix}{\operatorname{Mat}}
\newcommand{\SymmetricGroup}{\mathfrak{S}}

\newcommand{\localField}{F}
\newcommand{\ringOfIntegers}{\mathfrak{o}}

\newcommand{\maximalIdeal}{\mathfrak{p}}
\newcommand{\maximalCompact}{K}
\newcommand{\mdifferential}{\differential^{\times}}
\newcommand{\uniformizer}{\varpi}
\newcommand{\VolumeOf}{\operatorname{Vol}}

\newcommand{\fieldCharacterkc}[2]{\fieldCharacter_{\qty(#2^{#1})}}

\newcommand{\qAnalog}[2]{[#1]_{#2}}
\newcommand{\pHallLittlewood}{\mathrm{P}}
\newcommand{\qHallLittlewood}{\mathrm{Q}}
\newcommand{\hHallLittlewood}{\mathrm{H}}
\newcommand{\kostkaFoulkes}{\mathrm{K}}
\newcommand{\coKostkaFoulkes}{\tilde{\mathrm{K}}}
\newcommand{\htHallLittlewood}{\tilde{\mathrm{H}}}
\newcommand{\nOfPartition}{\mathfrak{n}}
\newcommand{\principalSeries}[1]{\operatorname{PS}(#1)}
\newcommand{\complexTuple}[1]{\underline{#1}}
\newcommand{\SphericalFunction}[1]{\varphi^{\circ}_{#1}}
\newcommand{\SpehSpherical}[2]{S^{\circ}_{#1, #2, \fieldCharacter}}
\newcommand{\rSpehSpherical}[2]{S_{#1, #2, \fieldCharacter}}
\newcommand{\WhittakerSpherical}[1]{W^{\circ}_{#1, \fieldCharacter}}
\newcommand{\SphericalSection}[1]{f^{\circ}_{\principalSeries{#1}}}

\newcommand{\zetaIntegral}{\operatorname{Z}}

\title[Casselman--Shalika formula for Speh representations]{On a Casselman--Shalika type formula for unramified Speh representations}

\author{Elad Zelingher}
\address{Department of Mathematics, University of Michigan, 1844 East Hall, 530 Church Street, Ann Arbor, MI 48109-1043 USA}
\email{eladz@umich.edu}

\hypersetup{pdfauthor={Elad Zelingher},
	pdfsubject={Representation theory},
	pdfkeywords={Whittaker functions, Speh representations, Hall--Littlewood polynomials}}

\keywords{Whittaker models, Casselman--Shalika, Speh representations}
\subjclass[2020]{Primary: 11F70. Secondary: 05E05, 11F66, 22E50, 33D52.}

\begin{document}

\begin{abstract}
	We give a Casselman--Shalika type formula for unramified Speh representations. Our formula computes values of the normalized spherical element of the $(k,c)$ model of a Speh representation at elements of the form $\diag\left(g, \IdentityMatrix{(k-1)c}\right)$, where $g \in \GL_c\left(\localField\right)$ for a non-archimedean local field $\localField$. The formula expresses these values in terms of modified Hall--Littlewood polynomials evaluated at the Satake parameter attached to the representation. Our proof is combinatorial and very simple. It utilizes Macdonald's formula and the unramified computation of the Ginzburg--Kaplan integral. This addresses a question of Lapid--Mao \cite{LapidMao2020}.
\end{abstract}

\maketitle

\section{Introduction}\label{sec:introduction}

In representation theory, models of representations are an important concept. The idea is quite simple: suppose that $G$ is a topological group and that $H$ is a closed subgroup of $G$. Let $\pi$ and $\sigma$ be irreducible representations of $G$ and $H$, respectively, such that $\Hom_{H}\left(\pi\restriction_H, \sigma\right)$ is one-dimensional. Then by Frobenius reciprocity there exists a unique subspace of $\Ind{H}{G}{\sigma}$ that is isomorphic to $\pi$. This is a model of $\pi$, i.e., a realization of $\pi$ as a space of functions from $G$ to the space of $\sigma$. Many areas in representation theory concern models of representations or rely on them. We briefly mention as an example the local Gan--Gross--Prasad conjectures, which concern Bessel and Fourier--Jacobi models \cite{GGP2012}.

In the recent past years, certain new models, which are called $(k,c)$ models, have been studied \cite{LapidMao2020, CaiFriedbergGinzburgKaplan2019, CaiFriedbergGourevitchKaplan2023}. Let $\localField$ be a non-archimedean local field, and let $\fieldCharacter \colon \localField \to \multiplicativegroup{\cComplex}$ be a non-trivial character. Given positive integers $k$ and $c$, let $\UnipotentRadicalForWss{k}{c}$ be the unipotent radical of the standard parabolic subgroup $\ParabolicForWss{k}{c}$ of $\GL_{kc}\left(\localField\right)$ corresponding to the composition $(c^k)$. We define a character $\fieldCharacterkc{k}{c} \colon \UnipotentRadicalForWss{k}{c} \to \multiplicativegroup{\cComplex}$ by the formula
$$\fieldCharacterkc{k}{c} \begin{pmatrix}
	\IdentityMatrix{c} & X_1 & \ast & \ast  & \ast \\
	& \IdentityMatrix{c} & X_2 & \ast & \ast \\
	& & \ddots & \ddots & \ast  \\
	& & & \IdentityMatrix{c} &  X_{k-1} \\
	& & & & \IdentityMatrix{c}
\end{pmatrix} = \fieldCharacter\left( \sum_{j=1}^{k-1} \trace X_j \right).$$
An irreducible representation $\pi$ of $\GL_{kc}\left(\localField\right)$ is said to have a $(k,c)$ model if there exists a unique subspace of $\Ind{\UnipotentRadicalForWss{k}{c}}{\GL_{kc}\left(\localField\right)}{\fieldCharacterkc{k}{c}}$ that is isomorphic to $\pi$, in which case we denote it $\Whittaker\left(\pi, \fieldCharacterkc{k}{c}\right)$ and call it the \emph{$(k,c)$ model of $\pi$}. Notice that a $(k,1)$ model of an irreducible representation $\pi$ of $\GL_k\left(\localField\right)$ is simply a Whittaker model of $\pi$, and recall that $\pi$ is called generic if it admits such model. Whittaker models are of great importance in the representation theory of $p$-adic groups, and have many applications, including integral constructions of local factors \cite{Jacquet1983rankin, shahidi1984fourier, shahidi1990proof} and indexing of $L$-packets \cite[Section 9, item (2)]{GGP2012}.

Certain special representations, called (generalized) Speh representations, have a $(k,c)$ model (this is proved in \cite{LapidMao2020} and \cite{CaiFriedbergGourevitchKaplan2023}). Speh representations are parameterized by tuples $(\tau, c)$, where $\tau$ is an irreducible generic representation of $\GL_k\left(\localField\right)$ and $c$ is a positive integer (see \cite{LapidMao2020} and \cite{CaiFriedbergGourevitchKaplan2023} for details). We denote the corresponding representation of $\GL_{kc}\left(\localField\right)$ by $\SpehRepresentation{\tau}{c}$.

In the past recent years, new integral representations for various $L$-functions appeared in the literature \cite{CaiFriedbergGinzburgKaplan2019,CaiFriedbergKaplan2022,kaplan2018, GourevitchKaplan2023,LapidMao2020,ginzburg2019tensor,GinzburgSoudry2020}. These constructions make use of $(k,c)$ models of Speh representations attached to an irreducible generic representation $\tau$ of $\GL_k\left(\localField\right)$. Among these constructions are the generalized doubling method \cite{CaiFriedbergGinzburgKaplan2019,CaiFriedbergKaplan2022,GourevitchKaplan2023,kaplan2018} and the Ginzburg--Kaplan integral \cite{ginzburg2019tensor,Kaplan2023}, \cite[Appendix A]{kaplan2018}, both of which are constructions for the $L$-function corresponding to the tensor product representation attached to a pair of representations $\pi$ and $\tau$ of $G$ and $\GL_k$, respectively, where for the generalized doubling method $G$ is a classical group, and for the Ginzburg--Kaplan integral $G = \GL_c$. An important feature of both of these constructions is that they only require the representation $\tau$ to be generic, and do not assume anything about the representation $\pi$. This allowed Cai--Friedberg--Kaplan to give a proof based on $L$-functions and converse theorems for the existence of the standard functorial transfer of cuspidal automorphic representations of split classical groups, without assuming that the representations in question are generic \cite{kaplan2018}.

Naturally, when working on integral representations for $L$-functions, knowledge of certain values of special functions attached to unramified data is of great importance for the purpose of unramified computations. When $\tau = \principalSeries{\complexTuple{z}}$ is an unramified irreducible representation of $\GL_k\left(\localField\right)$ satisfying certain conditions, the Speh representation $\SpehRepresentation{\tau}{c}$ is also an unramified irreducible representation, and there exists a unique spherical element $\SpehSpherical{\complexTuple{z}}{c}$ in $\Whittaker\left(\SpehRepresentation{\tau}{c}, \fieldCharacterkc{k}{c}\right)$, such that its value at identity is $1$. In this article, we find an explicit formula for the special values $\SpehSpherical{\complexTuple{z}}{c}\left(\begin{smallmatrix}
	g\\
	& \IdentityMatrix{\left(k-1\right)c}
\end{smallmatrix}\right)$, where $g \in \GL_c\left(\localField\right)$, and where $\fieldCharacter$ is a character with conductor $0$. We refer the reader to \Cref{thm:special-values-theorem-summary} for our result, which expresses these special values as modified Hall--Littlewood polynomials evaluated at the Satake parameters of $\tau = \principalSeries{\complexTuple{z}}$. This gives a partial answer to a question raised by Lapid--Mao in \cite[Section 7]{LapidMao2020}.

We remark that for the unramified computations done in the generalized doubling method and the Ginzburg--Kaplan integrals, a formula as the one we provide here was not available, and the authors were able to use inductive methods instead. In fact, our proof relies on the unramified computation for the Ginzburg--Kaplan integral done by Ginzburg in \cite{ginzburg2019tensor}. In \Cref{rem:unramified-computation-of-ginzburg-kaplan-integral-using-modified-hall-littlewood} we show that one can give a short proof of the unramified computation of this aforementioned integral using our formula.

Many formulas for values of spherical elements corresponding to different models of unramified representations have been given in the past. We mention Shintani's formula \cite{shintani1976explicit} for values of the normalized spherical element in the Whittaker model, which was later generalized by Casselman--Shalika \cite{CasselmanShalika1980} to arbitrary quasi-split connected reductive groups. Sakellaridis \cite{Sakellaridis2013} and Sakellaridis--Venkatesh \cite{SakellaridisVenkatesh2017} have developed a framework which provides formulas for values attached to spherical elements as such for spherical varieties. However, this framework is not applicable in our situation, since $(k,c)$ models do not fall into the framework of spherical varieties for $k, c > 1$, unless $k = 2$ or $\left(k, c\right) = \left(3, 2\right)$ (see the introduction of \cite{LapidMao2020}).

Our result is very similar to our recent results with Oded Carmon regarding Bessel--Speh functions for Speh representations over finite fields \cite{CarmonZelingher2024}, and our method of proof is similar to \cite{Zelingher2023}. This suggests that there is a deep connection between Speh representations and modified Hall--Littlewood polynomials. The proof of our main result (\Cref{thm:main-result}) is very short. The reason for this is because we rely on difficult computations done previously. Besides the unramified computation of the Ginzburg--Kaplan integral mentioned above (which in turn uses Shintani's formula \cite{shintani1976explicit}), we rely on Macdonald's formula \cite{Casselman1980} for spherical matrix coefficients and on the Cauchy identity for Schur functions. It would be interesting to give a more direct proof of this statement by showing that these special values of this spherical element satisfy certain recurrence relations satisfied by modified Hall--Littlewood polynomials. We also mention that when $c=1$, our result does not recover the full Shintani formula, but only a special case of the formula (see \Cref{rem:c-equals-1-shintani-formula-special-case}). It would be interesting to give a more general formula for general $k$ and $c$ that recovers the full Shintani formula for $c=1$. We plan to return to both of these points in the future.

\section{Preliminaries}\label{sec:preliminaries}

\subsection{Hall--Littlewood polynomials}\label{sec:hall-littlewood-polynomials}

The purpose of this section is to recall the definitions of several variants of Hall--Littlewood polynomials.

By a partition $\complexTuple{\lambda}$ of a non-negative integer $m$, we mean a weakly decreasing sequence of positive integers $\complexTuple{\lambda} = (\lambda_1,\dots,\lambda_{\ell})$ such that $\lambda_1 + \dots + \lambda_{\ell} = m$. The integer $\ell$ is called \emph{the length of the partition $\complexTuple{\lambda}$}. We write $\complexTuple{\lambda} \vdash m$ to indicate that $\complexTuple{\lambda}$ is a partition of $m$, and denote $\sizeof{\lambda} = m$.

Let $n$ be a positive integer, and let $\SymmetricGroup_n$ be the symmetric group on $n$ elements. The group $\SymmetricGroup_n$ acts on the space of polynomials in complex coefficients $\cComplex\left[\complexTuple{x}\right] = \cComplex\left[x_1,\dots,x_n\right]$ in the variables $\complexTuple{x} = \left(x_1,\dots,x_n\right)$ by permuting the variables. Consider the subspace $\cComplex\left[\complexTuple{x}\right]^{\SymmetricGroup_n}$ consisting of symmetric polynomials. 

Let $\complexTuple{\lambda} = \left(\lambda_1,\dots,\lambda_\ell\right) \vdash m$ be a partition such that $\ell \le n$. The \emph{Hall--Littlewood polynomial}  $\pHallLittlewood_{\complexTuple{\lambda}}\left(\complexTuple{x}; t\right)$ is defined by the formula (see \cite[Definition 3.4.1]{Haiman2003} and \cite[Page 208]{macdonald1998symmetric}) $$\pHallLittlewood_{\complexTuple{\lambda}}\left(x; t\right) = \frac{1}{\prod_{i=0}^{\infty} \qAnalog{\alpha_i}{t}!} \sum_{\sigma \in \SymmetricGroup_n} \sigma\left( \complexTuple{x}^{\complexTuple{\lambda}} \frac{\prod_{i < j} \left(1 - t \frac{x_j}{x_i}\right)}{\prod_{i < j} \left(1 - \frac{x_j}{x_i}\right)} \right).$$
Here, for $j \ge 1$, $\alpha_j = \complexTuple{\lambda}\left(j\right)$ is the number of times that $j$ appears in $\complexTuple{\lambda}$, and $\alpha_0$ is defined so that $\sum_{j = 0}^{\infty} \alpha_j = n$. The symbol $\complexTuple{x}^{\complexTuple{\lambda}}$ means $x_1^{\lambda_1} x_2^{\lambda_2} \dots x_\ell^{\lambda_\ell}$ and \begin{align*}
	 \qAnalog{n}{t}! = \prod_{j=1}^n \qAnalog{j}{t}\,\,\,\,\,\,\,\text{where}\,\,\,\,\,\,\,\qAnalog{k}{t} = \prod_{j = 1}^k \frac{1 - t^j}{1 - t}.
\end{align*}
We also denote $$\qHallLittlewood_{\complexTuple{\lambda}}\left(\complexTuple{x};t\right) = \left(\left(1 - t\right)^{\ell} \prod_{i=1}^{\infty} \qAnalog{\complexTuple{\lambda}\left(i\right)}{t}!\right) \pHallLittlewood_{\complexTuple{\lambda}}\left(\complexTuple{x};t\right).$$

For any partition $\complexTuple{\lambda}$, we may associate a symmetric polynomial $s_{\complexTuple{\lambda}}$ called the \emph{Schur polynomial} (see \cite[Page 40]{macdonald1998symmetric}). It is given by $s_{\complexTuple{\lambda}}\left(\complexTuple{x}\right) = \pHallLittlewood_{\complexTuple{\lambda}}\left(\complexTuple{x}; 0\right)$, where if $\complexTuple{\lambda}$ is a partition of length greater than $n$, then $s_{\complexTuple{\lambda}}\left(\complexTuple{x}\right) = 0$. Schur polynomials form a basis for the space $\cComplex\left[\complexTuple{x}\right]^{\SymmetricGroup_n}$. We equip this space with the bilinear pairing $\standardForm{\cdot}{\cdot}$ defined so that the Schur polynomials $s_{\complexTuple{\lambda}}$ are orthonormal with respect to this pairing, i.e., for any two partitions $\complexTuple{\lambda}$ and $\complexTuple{\mu}$, $$\standardForm{s_{\complexTuple{\lambda}}}{s_{\complexTuple{\mu}}} = \delta_{\complexTuple{\lambda},\complexTuple{\mu}} =\begin{dcases}
	0 & \text{if }\complexTuple{\lambda} \ne \complexTuple{\mu,}\\
	1 & \text{if }\complexTuple{\lambda} = \complexTuple{\mu}
\end{dcases},$$ where is $\delta_{\complexTuple{\lambda},\complexTuple{\mu}}$ Kronecker's delta function.

We will need the \emph{transformed Hall--Littlewood polynomial} $\hHallLittlewood_{\complexTuple{\mu}}\left(x;t\right)$ defined for a partition $\complexTuple{\mu}$  (see also \cite[Section 3.4.2]{Haiman2003}). For any $j \ge 1$ let $p_{j}$ be $j$-th power sum polynomial, i.e., $$p_{j}\left(\complexTuple{x}\right) = x_1^j + \dots + x_n^j.$$
The ring $\cComplex\left[\complexTuple{x}\right]^{\SymmetricGroup_n}$ is a free $\cComplex$-algebra with generators $\left(p_j\right)_{j=1}^{\infty}$. The transformed Hall--Littlewood polynomial $\hHallLittlewood_{\complexTuple{\mu}}\left(\complexTuple{x};t\right)$ is defined as the image of $\qHallLittlewood_{\complexTuple{\mu}}\left(\complexTuple{x};t\right)$ under the homomorphism of rings that maps $p_j$ to $\left(1-t^j\right)^{-1} p_j$ for every $j$. In plethystic notation,
$$\hHallLittlewood_{\complexTuple{\mu}}\left(\complexTuple{x};t\right) = \qHallLittlewood_{\complexTuple{\mu}}\left[\frac{X}{1-t};t\right].$$

By \cite[Corollary 3.4.12]{Haiman2003}, the expansion of the transformed Hall--Littlewood polynomial with respect to the basis of Schur polynomials is given by \begin{equation}\label{eq:expansion-of-transformed-hall-littlewood-in-schur-basis}
	\hHallLittlewood_{\complexTuple{\mu}}\left(\complexTuple{x};t\right) = \sum_{\complexTuple{\lambda}} \kostkaFoulkes_{\complexTuple{\lambda}, \complexTuple{\mu}}\left(t\right) s_{\complexTuple{\lambda}}\left(\complexTuple{x}\right),
\end{equation}
where $\complexTuple{\lambda}$ runs over all partitions of $m = \sizeof{\complexTuple{\mu}}$, and $\kostkaFoulkes_{\complexTuple{\lambda},\complexTuple{\mu}}\left(t\right)$ is the Kostka--Foulkes polynomial (see \cite[Page 242]{macdonald1998symmetric} and \cite[Section 3.4.1]{Haiman2003}). By \cite[Corollary 3.4.10]{Haiman2003}, we have that with respect to the pairing defined above, the bases $\pHallLittlewood_{\complexTuple{\lambda}}$ and $\hHallLittlewood_{\complexTuple{\mu}}$ are dual, that is,
$$\standardForm{\pHallLittlewood_{\complexTuple{\lambda}}}{\hHallLittlewood_{\complexTuple{\mu}}} = \delta_{\complexTuple{\lambda},\complexTuple{\mu}}.$$

We will also need the \emph{modified Hall--Littlewood polynomial} (see also \cite[Section 3.4.7]{Haiman2003}). It is defined by the formula
\begin{equation}\label{eq:modified-hall-littlewood-definition}
	\htHallLittlewood_{\complexTuple{\mu}}\left(\complexTuple{x}; t\right) = t^{\nOfPartition\left(\complexTuple{\mu}\right)} \hHallLittlewood_{\complexTuple{\mu}}\left(\complexTuple{x}; t^{-1}\right),
\end{equation}
where for a partition $\complexTuple{\lambda} = \left(\lambda_1,\dots,\lambda_{\ell}\right)$, $$\nOfPartition\left(\complexTuple{\lambda}\right) = \sum_{j=1}^{\ell} \left(j-1\right) \lambda_j.$$

\subsubsection{Formula for modified Hall--Littlewood polynomials in terms of flags fixed by Jordan matrices}

Let $q$ be a prime power. When $t = q$, we have a formula expressing $\htHallLittlewood_{\complexTuple{\mu}}\left(\complexTuple{x}; q\right)$ in terms of flags fixed by Jordan matrices. 

Let $\finiteField_q$ be a finite field with $q$ elements. Given a partition $\complexTuple{\mu} \vdash m$, let $J_{\complexTuple{\mu}}\left(1\right) \in \GL_m\left(\finiteField_q\right)$ be the unipotent Jordan matrix of type $\complexTuple{\mu}$ (with the convention that $J_{\left(1^m\right)}\left(1\right) = \IdentityMatrix{m}$). We have the following formula for the modified Hall--Littlewood polynomial (see \cite[Lemma 2.8]{Ram2023}):
$$\htHallLittlewood_{\complexTuple{\mu}}\left(\complexTuple{x}; q\right) = \sum_{\complexTuple{\lambda} \vdash m} \#\left\{ \mathcal{F} \text{ flag in } \finiteField_q^m \text{ of type } \complexTuple{\lambda} \mid J_{\complexTuple{\mu}}\left(1\right) \mathcal{F} = \mathcal{F} \right\} \cdot m_{\complexTuple{\lambda}}\left(\complexTuple{x}\right),$$
where $m_{\complexTuple{\lambda}}\left(\complexTuple{x}\right)$ is the monomial symmetric polynomial corresponding to $\complexTuple{\lambda}$.
\begin{remark}\label{rem:homogeneous-polynomial-case-hall-littlewood}
	In the special case where $\complexTuple{\mu} = \left(m\right)$ is a partition of length $1$, we have that for every $0 \le j \le m$ there exists a unique subspace of $\finiteField_q^m$ of dimension $j$ that is invariant under $J_{(m)}\left(1\right)$, and thus a unique flag of type $\complexTuple{\lambda}$ invariant under $J_{(m)}\left(1\right)$ for every partition $\complexTuple{\lambda}$. Hence, we have $$\htHallLittlewood_{\left(m\right)}\left(\complexTuple{x}; q\right) = \sum_{\complexTuple{\lambda} \vdash m} m_{\complexTuple{\lambda}}\left(\complexTuple{x}\right) = h_m\left( \complexTuple{x} \right),$$
	where $h_m$ is the complete homogeneous polynomial of degree $m$.
\end{remark}
\begin{remark}\label{rem:k-equals-one-case-hall-littlewood}
	In the special case where $n=1$, the only partition $\complexTuple{\lambda}$ for which $m_{\complexTuple{\lambda}}$ is non-zero is $\lambda = (m)$, where $\sizeof{\complexTuple{\mu}} = m$. Hence in this case, $$\htHallLittlewood_{\complexTuple{\mu}}\left(\complexTuple{x}; q\right) = m_{(\sizeof{\complexTuple{\mu}})}\left(x_1\right) = x_1^{\sizeof{\complexTuple{\mu}}}.$$
\end{remark}

\subsection{Unramified representations and spherical vectors}\label{sec:unramified-representations-and-spherical-vectors}

Let $\localField$ be a non-archimedean local field with ring of integers $\ringOfIntegers$, maximal ideal $\maximalIdeal \subset \ringOfIntegers$, and residue field of cardinality $q = \sizeof{\ringOfIntegers \slash \maximalIdeal}$. Let $\uniformizer \in \maximalIdeal$ be a uniformizer, and normalize the absolute value on $\localField$ such that $\abs{\uniformizer} = \frac{1}{q}$.

Throughout the article, when referring to representations of $\GL_n\left(\localField\right)$, we will always mean smooth admissible representations. Let $B_n \subset \GL_n\left(\localField\right)$ be the upper triangular Borel subgroup, and let $\maximalCompact_n = \GL_n\left(\ringOfIntegers\right)$.

Given complex numbers $z_1, \dots, z_n$, denote $\complexTuple{z} =\left(z_1,\dots,z_n\right)$. Consider the (normalized) parabolically induced representation $$\principalSeries{\complexTuple{z}} = \Ind{B_n}{\GL_n\left(\localField\right)}{\abs{\cdot}^{z_1} \boxtimes \dots \boxtimes \abs{\cdot}^{z_n}}.$$
We call $\principalSeries{\complexTuple{z}}$ an \emph{unramified principal series representation} with parameter $\complexTuple{z}$, or just an \emph{unramified representation} with parameter $\complexTuple{z}$.

By \cite[Section 4.2]{BernsteinZelevinsky1977}, the representation $\principalSeries{\complexTuple{z}}$ is irreducible if and only if $z_i \ne z_j + 1  \left(\bmod \frac{2 \pi i}{\log q} \zIntegers \right)$ for every $i,j$. The contragredient representation of $\principalSeries{\complexTuple{z}}$ is $\principalSeries{-\complexTuple{z}}$. The pairing $\principalSeries{\complexTuple{z}} \times \principalSeries{-\complexTuple{z}} \to \cComplex$ is realized via the formula $$\standardForm{f_{\principalSeries{\complexTuple{z}}}}{f_{\principalSeries{-\complexTuple{z}}}} = \int_{\maximalCompact_n} f_{\principalSeries{\complexTuple{z}}}\left(k_0\right) f_{\principalSeries{-\complexTuple{z}}}\left(k_0\right) \mdifferential k_0,$$
where we normalize the Haar measure so that $\maximalCompact_n$ has volume $1$.

An irreducible representation $\pi$ of $\GL_n\left(\localField\right)$ is called \emph{spherical} if it admits a non-zero $\maximalCompact_n$-fixed vector, i.e., if there exists $0 \ne v \in \pi$ such that $\pi\left(k_0\right) v = v$ for every $k_0 \in \maximalCompact_n$. Such vector is called a \emph{spherical vector}. It is well-known that if $\pi$ is spherical, then its spherical vector is unique, up to scalar multiplication. It is also well-known that if $\pi$ is spherical, then $\pi$ can be realized as a subrepresentation of $\principalSeries{\complexTuple{z}}$ for some $\complexTuple{z}$.

By the Iwasawa decomposition, the representation $\principalSeries{\complexTuple{z}}$ always admits a unique spherical vector with the property that its value at $\IdentityMatrix{n}$ is $1$. We denote this element by $\SphericalSection{\complexTuple{z}}$.

\subsubsection{Macdonald's formula for spherical matrix coefficients}
Given a choice of $\complexTuple{z}$, we define the \emph{normalized spherical matrix coefficient} by the formula $$\SphericalFunction{\complexTuple{z}}\left(g\right) = \standardForm{\principalSeries{\complexTuple{z}}\left(g\right) \SphericalSection{\complexTuple{z}}}{\SphericalSection{-\complexTuple{z}}},$$
where $g \in \GL_n\left(\localField\right)$. It follows from the definitions that $\SphericalFunction{\complexTuple{z}}$ is $\maximalCompact_n$-bi-invariant and that it is $\left(\multiplicativegroup{\localField}, \abs{\cdot}^{\sum_{i=1}^n z_i}\right)$ bi-equivariant (where $\multiplicativegroup{\localField}$ is realized as a subgroup of $\GL_n\left(\localField\right)$ via the embedding $x \mapsto x \IdentityMatrix{n}$). By the Cartan decomposition, it suffices to know the values of $\SphericalFunction{\complexTuple{z}}$ on elements of the form $\uniformizer^{\left(i_1,\dots,i_n\right)} \coloneqq \diag\left(\uniformizer^{i_1},\dots,\uniformizer^{i_n}\right),$ where $i_1 \ge \dots \ge i_n \ge 0$. For a partition $\complexTuple{\lambda} = \left(\lambda_1,\dots,\lambda_\ell\right)$ of length $\ell \le n$, we define $\uniformizer^{\complexTuple{\lambda}} = \uniformizer^{(\lambda_1,\dots,\lambda_{\ell})}$ and $\uniformizer^{\complexTuple{\lambda}, n} = \left(\begin{smallmatrix}
	\uniformizer^{\complexTuple{\lambda}}\\
	& \IdentityMatrix{n-\ell}
\end{smallmatrix}\right)$. The following formula for the value $\SphericalFunction{\complexTuple{z}}\left(\uniformizer^{\complexTuple{\lambda}, n}\right)$ is due to Macdonald \cite[Page 299]{macdonald1998symmetric}, \cite{Casselman1980}. Recall that we normalized the Haar measure so that $\maximalCompact_n$ has volume $1$.
\begin{theorem}\label{thm:macdonald-formula}
	For any partition $\complexTuple{\lambda}$ of length $\ell \le n$ and any choice of $\complexTuple{z} =\left(z_1,\dots,z_n\right)$, we have
	$$\VolumeOf\left( \maximalCompact_n \uniformizer^{\complexTuple{\lambda}, n} \maximalCompact_n \right) \cdot \SphericalFunction{\complexTuple{z}}\left(\uniformizer^{\complexTuple{\lambda}}\right) = q^{\frac{\left(n-1\right) \cdot \sizeof{\complexTuple{\lambda}}}{2} -\nOfPartition\left(\complexTuple{\lambda}\right)} \cdot \pHallLittlewood_{\complexTuple{\lambda}}\left(q^{-z_1},\dots,q^{-z_n}; q^{-1}\right).$$
\end{theorem}

\subsection{\texorpdfstring{$(k,c)$}{(k,c)} functionals}

Let $\fieldCharacter \colon \localField \to \multiplicativegroup{\cComplex}$ be a non-trivial character with conductor $\ringOfIntegers$, i.e., $\fieldCharacter$ is trivial on $\ringOfIntegers$ but not on $\uniformizer^{-1} \ringOfIntegers$. Let $k$ and $c$ be positive integers. Let $\UnipotentRadicalForWss{k}{c}$ be the unipotent radical of the standard parabolic subgroup $\ParabolicForWss{k}{c}$ of $\GL_{kc}\left(\localField\right)$ corresponding to the composition $(c^k)$. We define a character $\fieldCharacterkc{k}{c} \colon \UnipotentRadicalForWss{k}{c} \to \multiplicativegroup{\cComplex}$ by the formula $$\fieldCharacterkc{k}{c} \begin{pmatrix}
	\IdentityMatrix{c} & X_1 & \ast & \ast  & \ast \\
	& \IdentityMatrix{c} & X_2 & \ast & \ast \\
	& & \ddots & \ddots & \ast  \\
	& & & \IdentityMatrix{c} &  X_{k-1} \\
	& & & & \IdentityMatrix{c}
\end{pmatrix} = \fieldCharacter\left( \sum_{j=1}^{k-1} \trace X_j \right).$$
Let $\pi$ be an irreducible representation of $\GL_{kc}\left(\localField\right)$. A \emph{$(k,c)$ functional of $\pi$} is a non-zero functional $\Lambda$ on the space of $\pi$ such that for any $u \in \UnipotentRadicalForWss{k}{c}$ and any $v \in \pi$, $$\standardForm{\pi(u) v}{\Lambda} = \fieldCharacterkc{k}{c}\left(u\right) \standardForm{v}{\Lambda}.$$

Suppose that $\pi$ admits a $(k,c)$ functional, and that its $(k,c)$ functional is unique (up to scalar multiplication). By Frobenius reciprocity there exists a unique subspace $\Whittaker\left(\pi, \fieldCharacterkc{k}{c}\right)$ of $\Ind{\UnipotentRadicalForWss{k}{c}}{\GL_{kc}\left(\localField\right)}{\fieldCharacterkc{k}{c}}$ that is isomorphic to $\pi$. We call $\Whittaker\left(\pi, \fieldCharacterkc{k}{c}\right)$ \emph{the $(k,c)$ model of $\pi$}.

\subsubsection{Generic representations}
When $c=1$, we denote $U_k = \UnipotentRadicalForWss{k}{1}$ and $\fieldCharacter = \fieldCharacterkc{k}{1}$. In this case, we simply say that $\Lambda$ is a $\fieldCharacter$-Whittaker functional, instead of saying that $\Lambda$ is a $(k, 1)$ functional. An irreducible representation $\pi$ is called \emph{generic} if it admits a $\fieldCharacter$-Whittaker functional, in which case it is well-known that it is unique (up to scalar multiplication). This notion is independent of the choice of the additive character $\fieldCharacter \colon \localField \to \multiplicativegroup{\cComplex}$.

\subsubsection{Shintani and Casselman--Shalika formula}
By \cite[Theorem 4.11]{BernsteinZelevinsky1977}, for any $\complexTuple{z} =\left(z_1,\dots,z_k\right)$ such that the representation $\principalSeries{\complexTuple{z}}$ is irreducible, $\principalSeries{\complexTuple{z}}$ is generic. Given such $\complexTuple{z}$, it turns out that any spherical vector in $\Whittaker\left(\principalSeries{\complexTuple{z}}, \fieldCharacter\right)$ does not vanish at $\IdentityMatrix{k}$. Let $\WhittakerSpherical{\complexTuple{z}} \in \Whittaker\left(\principalSeries{\complexTuple{z}}, \fieldCharacter\right)$ be the unique spherical vector with value $1$ at $\IdentityMatrix{k}$. Since $\WhittakerSpherical{\complexTuple{z}}\left(g\right)$ is $\left(\UnipotentSubgroup_k \rtimes \multiplicativegroup{\localField}, \fieldCharacter \otimes \abs{\cdot}^{\sum_{i=1}^k z_i}\right)$-equivariant from the left and $\maximalCompact_k$-invariant from the right, by the Iwasawa decomposition it suffices to know the values of $\WhittakerSpherical{\complexTuple{z}}$ on elements of the form $\uniformizer^{(i_1,\dots,i_k)}$ where $i_k \ge 0$. Such a formula was given by Shintani \cite{shintani1976explicit} and then generalized by Casselman--Shalika \cite{CasselmanShalika1980} for arbitrary quasi-split connected reductive groups.

\begin{theorem}\label{thm:shintani-formula}
	Let $\complexTuple{z}$ be such that $z_i \ne z_j + 1  \left(\bmod \frac{2 \pi i}{\log q} \zIntegers \right)$. The value $\WhittakerSpherical{\complexTuple{z}}\left(\uniformizer^{(i_1,\dots,i_k)}\right)$ is zero unless $i_1 \ge \dots \ge i_k$. For any partition $\complexTuple{\lambda}$ of length $\ell \le k$ 
	$$\WhittakerSpherical{\complexTuple{z}}\left(\uniformizer^{\complexTuple{\lambda}, k}\right) = \delta^{\frac{1}{2}}_{B_k}\left(\uniformizer^{\complexTuple{\lambda}, k}\right) \cdot s_{\complexTuple{\lambda}}\left(q^{-z_1},\dots,q^{-z_k}\right).$$
\end{theorem}

\subsection{Speh representations of unramified principal series representations}
Let $\tau$ be an irreducible generic representation of $\GL_k\left(\localField\right)$. In \cite{CaiFriedbergGourevitchKaplan2023}, Cai--Friedberg--Gourevitch--Kaplan defined the notion of the generalized Speh representation $\SpehRepresentation{\tau}{c}$, for any $c \ge 1$. They proved that the representation $\SpehRepresentation{\tau}{c}$ is a $(k,c)$ representation, which in particular means that it admits a unique (up to scalar multiplication) $(k,c)$ functional. They proved in \cite[Lemma 12]{CaiFriedbergGourevitchKaplan2023} that for any $W \in \Whittaker\left(\SpehRepresentation{\tau}{c}, \fieldCharacterkc{k}{c}\right)$, any $g \in \GL_c\left(\localField\right)$ and any $h \in \GL_{kc}\left(\localField\right)$, $$W\left(\diag^{k}\left(g\right) h\right)  = \centralCharacter{\tau}\left(\det g\right) W\left(h\right),$$
where $\centralCharacter{\tau}$ is the central character of $\tau$ and $$\diag^k\left(g\right) = \diag \left(g,\dots,g\right) \in \GL_{kc}\left(\localField\right).$$

We will focus on the case of \emph{irreducible unramified Speh representations}, which we explain now. Let $\tau = \principalSeries{\complexTuple{z}}$, where \begin{equation}\label{eq:condition-for-speh-representation-to-be-irreducible}
	z_i - z_j \ne \pm 1, \pm 2,\dots, \pm c \left(\bmod \frac{2 \pi i}{\log q} \zIntegers \right)
\end{equation} for every $i, j$ (notice that from \Cref{sec:unramified-representations-and-spherical-vectors}, this implies in particular that $\principalSeries{\complexTuple{z}}$ is irreducible). In this case, the Speh representation $\SpehRepresentation{\principalSeries{\complexTuple{z}}}{c}$ is the parabolically induced representation \begin{equation}\label{eq:speh-representation-as-parabolically-induced}
\SpehRepresentation{\principalSeries{\complexTuple{z}}}{c} = \Ind{\ParabolicForWss{k}{c}}{\GL_{kc}\left(\localField\right)}{\abs{{\det}_{\GL_c} }^{z_1} \boxtimes \dots \boxtimes \abs{{\det}_{\GL_c} }^{z_k}}.
\end{equation}
This coincides with the definition in  \cite{CaiFriedbergGourevitchKaplan2023} (see the paragraph just before Theorem 4 in \cite{CaiFriedbergGourevitchKaplan2023} and see the explanation below for independence of the order of $\complexTuple{z}$). It will follow from our explanation below that it also coincides with the definition in \cite{LapidMao2020}.

Zelevinsky's results \cite{Zelevinsky1980} imply that $\SpehRepresentation{\principalSeries{\complexTuple{z}}}{c}$ is irreducible. Indeed, in Zelevinsky's notations, $\SpehRepresentation{\principalSeries{\complexTuple{z}}}{c} = \left\langle \Delta_1 \right\rangle \times \dots \times \left\langle \Delta_k \right\rangle$ where (see \cite[Example 3.2]{Zelevinsky1980}) \begin{equation*}
	\Delta_i = \left\{\abs{\cdot}^{z_i - \frac{c-1}{2}},  \abs{\cdot}^{z_i - \frac{c-3}{2}}, \dots \abs{\cdot}^{z_i + \frac{c-1}{2}}\right\}.
\end{equation*}
Condition \eqref{eq:condition-for-speh-representation-to-be-irreducible} is equivalent (in Zelevinsky's terminology) to the segments $\left(\Delta_i\right)_{i=1}^k$ being mutually unlinked. Hence from \cite[Theorem 4.2]{Zelevinsky1980}, the Speh representation $\SpehRepresentation{\principalSeries{\complexTuple{z}}}{c}$ in \eqref{eq:speh-representation-as-parabolically-induced} is irreducible if and only if \eqref{eq:condition-for-speh-representation-to-be-irreducible} holds, and by \cite[Section 6.4]{Zelevinsky1980}, given \eqref{eq:condition-for-speh-representation-to-be-irreducible}, the equivalence class of \eqref{eq:speh-representation-as-parabolically-induced} does not depend on the order of the tuple $\complexTuple{z}$.
\begin{remark}
	One may drop the condition \eqref{eq:condition-for-speh-representation-to-be-irreducible} and define the notion of a Speh representation $\SpehRepresentation{\principalSeries{\complexTuple{z}}}{c}$ for more general irreducible generic unramified representations $\principalSeries{\complexTuple{z}}$ as in \cite{LapidMao2020}\footnote{This is slightly different from \cite{CaiFriedbergGourevitchKaplan2023}, as the Speh representation in \cite{CaiFriedbergGourevitchKaplan2023} need not be irreducible.}. By \cite[Remark 2.6]{LapidMao2020}, the representation $\SpehRepresentation{\principalSeries{\complexTuple{z}}}{c}$ is unramified exactly when the segments $\left(\Delta_i\right)_{i=1}^k$ are mutually unlinked, which is equivalent to condition \eqref{eq:condition-for-speh-representation-to-be-irreducible}.
\end{remark}

As in \Cref{sec:unramified-representations-and-spherical-vectors}, by the Iwasawa decomposition, the representation $\SpehRepresentation{\principalSeries{\complexTuple{z}}}{c}$ admits a unique spherical vector (up to scalar multiplication).

\subsubsection{Special values of the spherical $(k,c)$ function}\label{subsec:special-values-of-spherical-k-c-function}

Similarly to before, any spherical vector of $\Whittaker\left(\SpehRepresentation{\principalSeries{\complexTuple{z}}}{c}, \fieldCharacterkc{k}{c}\right)$ does not vanish at $\IdentityMatrix{kc}$ (this follows from the proof of \cite[Corollary 23]{CaiFriedbergGinzburgKaplan2019}, since we assume that \eqref{eq:condition-for-speh-representation-to-be-irreducible} holds). Let $\SpehSpherical{\complexTuple{z}}{c} \in \Whittaker\left(\SpehRepresentation{\principalSeries{\complexTuple{z}}}{\fieldCharacterkc{k}{c}}\right)$ be the unique spherical vector such that its value at $\IdentityMatrix{kc}$ is $1$.

The function $\SpehSpherical{\complexTuple{z}}{c}\left(h\right)$ is $\left(\UnipotentRadicalForWss{k}{c} \rtimes \diag^k\left(\GL_c\left(\localField\right)\right), \fieldCharacterkc{k}{c} \otimes \abs{\det_{\GL_c}}^{\sum_{i=1}^k z_i}\right)$-equivariant from the left and $\maximalCompact_{kc}$-invariant from the right (here $\abs{\det_{\GL_c}}$ is realized as a character of $\diag^k\left(\GL_c\left(\localField\right)\right)$ via the isomorphism $\GL_c\left(\localField\right) \ni g \mapsto \diag^k\left(g\right)$).

Unlike the previous cases where we were able to reduce the evaluation of special functions to their values on diagonal matrices, this case is more difficult. Thus, we will instead concentrate on matrices of a special form. Let $\rSpehSpherical{\complexTuple{z}}{c} \colon \GL_{c}\left(\localField\right) \to \cComplex$ be the function $$\rSpehSpherical{\complexTuple{z}}{c}\left(g\right) = \SpehSpherical{\complexTuple{z}}{c}\left(\diag\left(g, \IdentityMatrix{\left(k-1\right)c}\right)\right),$$
where $g \in \GL_{c}\left(\localField\right)$. This function is $\maximalCompact_c$-bi-invariant: it suffices to check the invariance from the left, and indeed for $k_1 \in \maximalCompact_c$ and $g \in \GL_c\left(\localField\right)$, $$\rSpehSpherical{\complexTuple{z}}{c}\left(k_1 g \right) = \abs{\det k_1}^{z_1 + \dots + z_k} \SpehSpherical{\complexTuple{z}}{c}\left(\diag\left(g, \IdentityMatrix{\left(k-1\right)c}\right) \diag\left(\IdentityMatrix{c}, \diag^{k-1}\left(k_1^{-1}\right)\right)\right) = \rSpehSpherical{\complexTuple{z}}{c}\left(g\right).$$
Hence, by the Cartan decomposition, it suffices to know the values of $\rSpehSpherical{\complexTuple{z}}{c}$ on matrices of the form $\uniformizer^{(i_1,\dots,i_c)}$, where $i_1 \ge \dots \ge i_c$. We claim that we may assume that $i_c \ge 0$.

\begin{proposition}\label{prop:vanishing-of-spherical-whittaker-speh-function}
	If $i_c < 0$, then $\rSpehSpherical{\complexTuple{z}}{c}\left( \uniformizer^{(i_1,\dots,i_c)}\right) = 0$.
\end{proposition}
\begin{proof}
	Omitted. See for example \cite[Section 6.2]{JacquetRallis1996}.
\end{proof}

Although the restriction to matrices of the form $\diag\left(g, \IdentityMatrix{\left(k-1\right)c}\right)$ might seem limiting, the knowledge of these values is sufficient for many applications. See for example \cite[Proposition 2.5]{CaiFriedbergKaplan2022}, \cite[Proposition 3.17]{kaplan2018} and \cite{Luo2024}. As explained in \cite[Page 1100]{Sakellaridis2006}, when $k=2$, it suffices to restrict to matrices of these form. For general $k$, if one wishes to compute $\SpehSpherical{\complexTuple{z}}{c}$ for general matrices, it is not sufficient to restrict to diagonal matrices. One would have to consider matrices of the form $\diag\left(g_1,\dots,g_k\right)$, where $g_1,\dots,g_k \in B_c$ are upper triangular, where it suffices to assume that $g_1$ is a diagonal matrix and $g_k = \IdentityMatrix{c}$. Alternatively, one can consider matrices of the form $\diag\left(l_1 a_1, \dots, l_k a_k\right)$, where for every $i$, $a_i$ is a diagonal matrix and $l_i \in \maximalCompact_c$, such that $a_{i+1}^{-1} l_{i+1}^{-1} l_i a_{i} \in \squareMatrix_c\left(\ringOfIntegers\right)$ for every $1 \le i \le k-1$, and it suffices to assume that $l_1 = l_k = a_k = \IdentityMatrix{c}$. In both of these two cases, much complexity is added.

\subsection{Ginzburg--Kaplan integrals}

We now describe a recent integral construction of doubling type for the tensor product $L$-function. Is is due to Ginzburg \cite{ginzburg2019tensor} in the global case, and due to Kaplan \cite[Appendix A]{kaplan2018} in the local case (see also \cite{Kaplan2023}).

Assume from now on that $k > 1$. Let $\pi$ and $\tau$ be irreducible representations of $\GL_c\left(\localField\right)$ and $\GL_k\left(\localField\right)$, respectively. Assume that $\tau$ is generic.

For any $s \in \cComplex$, $v \in \pi$, $v^{\vee} \in \Contragradient{\pi}$ and $W \in \Whittaker\left(\SpehRepresentation{\tau}{c}, \fieldCharacterkc{k}{c}\right)$, define the zeta integral
$$\zetaIntegral\left(s, v, v^{\vee}, W ; \fieldCharacter\right) = \int_{\GL_c\left(\localField\right)} W\left(\diag\left(g, \IdentityMatrix{\left(k-1\right)c}\right)\right) \standardForm{\pi \left(g\right) v}{v^{\vee}} \abs{\det g}^{s - \frac{1 + \left(k-2\right)c}{2}} \mdifferential g.$$

This integral converges absolutely for $\RealPart s$ large (depending only on the representations $\pi$ and $\tau$). For any choice of $v,v^{\vee}$ and $W$, the integral $\zetaIntegral\left(s, v, v^{\vee}, W ; \fieldCharacter\right)$ converges in its convergence domain to a rational function in $q^{-s}$ (an element of $\cComplex(q^{-s})$), and therefore admits a meromorphic continuation to the entire plane. We continue denoting the meromorphic continuation by the same symbol.

As before, we focus on the case where $\pi$ and $\tau$ are unramified principal series representations. Let $\complexTuple{z}' =(z'_1, \dots, z'_c)$ and let $\complexTuple{z} =\left(z_1,\dots,z_k\right)$ be tuples of complex numbers, and let $\pi = \principalSeries{\complexTuple{z}'}$ and $\tau = \principalSeries{\complexTuple{z}}$. Assume that $z_i - z_j \ne \pm 1, \pm 2,\dots,\pm c \left(\bmod \frac{2 \pi i}{\log q} \zIntegers \right)$ and $z'_i \ne z'_j + 1 \left(\bmod \frac{2 \pi i}{\log q} \zIntegers \right)$ for any $i$ and $j$. In \cite[Section 3]{ginzburg2019tensor}, Ginzburg showed the following result.

\begin{theorem}\label{thm:ginzburg-unramified-computation}
	Normalize the Haar measure so that $\maximalCompact_c$ has volume $1$. Then $$\zetaIntegral\left(s, \SphericalSection{\complexTuple{z}'}, \SphericalSection{-\complexTuple{z}'}, \SpehSpherical{\complexTuple{z}}{c} ; \fieldCharacter \right) = \prod_{i = 1}^c \prod_{j = 1}^k \frac{1}{1 - q^{-s} \cdot q^{-z'_i} \cdot q^{-z_j}} = L\left(s, \principalSeries{\complexTuple{z'}} \times \principalSeries{\complexTuple{z}}\right).$$
\end{theorem}

\section{Main result}

We use a combinatorial argument to compute the values of $\rSpehSpherical{\complexTuple{z}}{c}$ on diagonal matrices.

\begin{theorem}\label{thm:main-result}
	Let $\complexTuple{\mu}$ be a partition with of length $\ell \le c$. Then $$\rSpehSpherical{\complexTuple{z}}{c}\left(\uniformizer^{\complexTuple{\mu}, c}\right) = \delta_{\ParabolicForWss{k}{c}}^{\frac{1}{2}}\left(\uniformizer^{\complexTuple{\mu}, kc}\right) \cdot \htHallLittlewood_{\complexTuple{\mu}}\left(q^{-z_1},\dots,q^{-z_k} ;q\right).$$
\end{theorem}
\begin{proof}
	Using the Cartan decomposition and \Cref{prop:vanishing-of-spherical-whittaker-speh-function}, we rewrite the Ginzburg--Kaplan zeta integral $\zetaIntegral\left(s, \SphericalSection{\complexTuple{z}'}, \SphericalSection{-\complexTuple{z}'}, \SpehSpherical{\complexTuple{z}}{c} ; \fieldCharacter \right)$ as the sum
	\begin{equation}\label{eq:ginzburg-kaplan-zeta-integral-with-double-cosets}
		\sum_{i_1 \ge i_2 \ge \dots \ge i_c \ge 0} \int_{\maximalCompact_c \uniformizer^{(i_1,\dots,i_c)} \maximalCompact_c} \SphericalFunction{\complexTuple{z}'}\left(g\right) \rSpehSpherical{\complexTuple{z}}{c}\left(g\right) q^{\left(-s + \frac{1+(k-2)c}{2}\right) \left(i_1 + \dots + i_c\right)} \mdifferential g.
	\end{equation}
	Since both $\SphericalFunction{\complexTuple{z}'}$ and $\rSpehSpherical{\complexTuple{z}}{c}$ are constant on the double coset $\maximalCompact_c \uniformizer^{(i_1,\dots,i_c)} \maximalCompact_c$, we may rewrite \eqref{eq:ginzburg-kaplan-zeta-integral-with-double-cosets} as
	\begin{equation}\label{eq:ginzburg-kaplan-zeta-integral-with-volume}
		\sum_{\complexTuple{\lambda}} \VolumeOf\left(\maximalCompact_c \uniformizer^{\complexTuple{\lambda}, c} \maximalCompact_c\right) \SphericalFunction{\complexTuple{z}'}\left(\uniformizer^{\complexTuple{\lambda}, c}\right) \rSpehSpherical{\complexTuple{z}}{c}\left(\uniformizer^{\complexTuple{\lambda}, c}\right) q^{(-s + \frac{1 + (k-2)c}{2}) \sizeof{\complexTuple{\lambda}}},
	\end{equation}
	where $\complexTuple{\lambda}$ runs over all the partitions of length $\le c$.
	Denote $\complexTuple{x} = \left(x_1,\dots,x_c\right) = \left(q^{-z'_1}, \dots, q^{-z'_c}\right)$, $\complexTuple{y} = \left(y_1,\dots,y_k\right) = \left(q^{-z_1}, \dots, q^{-z_k}\right)$ and $T = q^{-s}$. By Macdonald's formula (\Cref{thm:macdonald-formula}) and by \eqref{eq:ginzburg-kaplan-zeta-integral-with-volume}, we have that \begin{equation}\label{eq:evaluation-of-unramified-data}
		\zetaIntegral\left(s, \SphericalSection{\complexTuple{z}'}, \SphericalSection{-\complexTuple{z}'}, \SpehSpherical{\complexTuple{z}}{c} ; \fieldCharacter \right) = \sum_{\complexTuple{\lambda}} q^{\frac{c \left(k-1\right) \cdot \abs{\complexTuple{\lambda}}}{2} - \nOfPartition\left(\complexTuple{\lambda}\right)} \pHallLittlewood_{\complexTuple{\lambda}}\left(\complexTuple{x}; q^{-1}\right) \rSpehSpherical{\complexTuple{z}}{c}\left(\uniformizer^{\complexTuple{\lambda}, c}\right) T^{\sizeof{\complexTuple{\lambda}}}.
	\end{equation}
	On the other hand, by the unramified computation of Ginzburg (\Cref{thm:ginzburg-unramified-computation}), we have
	$$\zetaIntegral\left(s, \SphericalSection{\complexTuple{z}'}, \SphericalSection{-\complexTuple{z}'}, \SpehSpherical{\complexTuple{z}}{c} ; \fieldCharacter \right) = \prod_{i = 1}^c \prod_{j = 1}^k \frac{1}{1 - T x_i y_j}.$$
	By the Cauchy identity for Schur polynomials (see \cite[Page 63]{macdonald1998symmetric}), \begin{equation}\label{eq:cauchy-identity-for-schur-polynomials}
		\prod_{i = 1}^c \prod_{j = 1}^k \frac{1}{1 - T x_i y_j} = \sum_{\complexTuple{\lambda}} s_{\complexTuple{\lambda}}\left(\complexTuple{x}\right) s_{\complexTuple{\lambda}}\left(\complexTuple{y}\right) T^{\sizeof{\complexTuple{\lambda}}},
	\end{equation}
	where $\complexTuple{\lambda}$ goes over all partitions.
	
	Hence, for any non-negative integer $m$, we have the equality
	\begin{equation}\label{eq:schur-convolution-equals-hall-littlewood-convolution}
		\sum_{\complexTuple{\lambda} \vdash m} s_{\complexTuple{\lambda}}\left(\complexTuple{x}\right) s_{\complexTuple{\lambda}}\left(\complexTuple{y}\right) =  \sum_{\complexTuple{\lambda} \vdash m} q^{\frac{c \left(k-1\right) \cdot \abs{\complexTuple{\lambda}}}{2} - \nOfPartition\left(\complexTuple{\lambda}\right)} \cdot \pHallLittlewood_{\complexTuple{\lambda}}\left(\complexTuple{x}; q^{-1}\right) \rSpehSpherical{\complexTuple{z}}{c}\left(\uniformizer^{\complexTuple{\lambda}, c}\right),
	\end{equation}
	where on the right hand side we sum only over partitions of length $\le c$. Let $m = \sizeof{\complexTuple{\mu}}$. Since the basis of transformed Hall--Littlewood polynomials $\left(\hHallLittlewood_{\complexTuple{\mu}}\left(\complexTuple{x};q^{-1}\right)\right)_{\complexTuple{\mu}}$ of the space $\cComplex\left[\complexTuple{x}\right]^{\SymmetricGroup_c}$ is dual to the basis of Hall--Littlewood polynomials  $\left(\pHallLittlewood_{\complexTuple{\lambda}}\left(\complexTuple{x}; q^{-1}\right)\right)_{\complexTuple{\lambda}}$, we get by pairing \eqref{eq:schur-convolution-equals-hall-littlewood-convolution} with $q^{-\frac{m c \left(k-1\right)}{2} + \nOfPartition\left(\mu\right)} \hHallLittlewood_{\complexTuple{\mu}}\left(\complexTuple{x}; q^{-1}\right)$ that
	$$\rSpehSpherical{\complexTuple{z}}{c}\left(\uniformizer^{\complexTuple{\mu}, c}\right) = q^{-\frac{c\left(k-1\right) \cdot m}{2}} \cdot q^{\nOfPartition\left(\complexTuple{\mu}\right)} \sum_{\complexTuple{\lambda} \vdash m} \standardForm{ s_{\complexTuple{\lambda}}\left(\complexTuple{x}\right)}{\hHallLittlewood_{\complexTuple{\mu}}\left(\complexTuple{x}; q^{-1}\right)} \cdot s_{\complexTuple{\lambda}}\left(\complexTuple{y}\right),$$
	where we regard $y_1,\dots,y_k$ as constants.
	Notice that by \eqref{eq:expansion-of-transformed-hall-littlewood-in-schur-basis}, we have 
	$$\standardForm{ s_{\complexTuple{\lambda}}\left(\complexTuple{x}\right)}{\hHallLittlewood_{\complexTuple{\mu}}\left(\complexTuple{x}; q^{-1}\right)} = \kostkaFoulkes_{\complexTuple{\lambda},\complexTuple{\mu}}\left(q^{-1}\right).$$
	Thus, \begin{equation}\label{eq:special-value-with-kostka-foulkes-polynomials}\rSpehSpherical{\complexTuple{z}}{c}\left(\uniformizer^{\complexTuple{\mu}, c}\right) = q^{-\frac{c\left(k-1\right) \cdot m}{2}} \cdot q^{\nOfPartition\left(\complexTuple{\mu}\right)} \sum_{\complexTuple{\lambda} \vdash m} \kostkaFoulkes_{\complexTuple{\lambda}, \complexTuple{\mu}}\left(q^{-1}\right) s_{\complexTuple{\lambda}}\left(\complexTuple{y}\right).
	\end{equation}
	Using \eqref{eq:expansion-of-transformed-hall-littlewood-in-schur-basis} again, we have that the right-hand side of \eqref{eq:special-value-with-kostka-foulkes-polynomials} is $$q^{-\frac{c(k-1) \cdot m}{2}} \cdot q^{\nOfPartition\left(\complexTuple{\mu}\right)} \cdot \hHallLittlewood_{\complexTuple{\mu}}\left(\complexTuple{y}; q^{-1}\right) = q^{-\frac{c(k-1) \cdot \sizeof{\complexTuple{\mu}}}{2}} \cdot \htHallLittlewood_{\complexTuple{\mu}}\left(\complexTuple{y}; q\right),$$
	as required.
\end{proof}
\begin{remark}\label{rem:c-equals-1-shintani-formula-special-case}
	For $c = 1$, in which case we consider the normalized spherical element in $\Whittaker\left( \principalSeries{\complexTuple{z}}, \fieldCharacter \right)$, the formula reads 	$$\WhittakerSpherical{\complexTuple{z}}\begin{pmatrix}
		\uniformizer^{m}\\
		& \IdentityMatrix{k - 1}
	\end{pmatrix} = \delta^{\frac{1}{2}}_{B_k}\begin{pmatrix}
	\uniformizer^m\\
	& \IdentityMatrix{k-1}
	\end{pmatrix} \htHallLittlewood_{(m)}\left(q^{-z_1},\dots,q^{-z_k};q\right),$$
	where $m \ge 0$.
	By \Cref{rem:homogeneous-polynomial-case-hall-littlewood}, This coincides with Shintani's formula (\Cref{thm:shintani-formula}), since the Schur polynomial $s_{(m)}$ coincides with $h_m$, the complete homogeneous polynomial of degree $m$.
\end{remark}
\begin{remark}
	By \Cref{rem:k-equals-one-case-hall-littlewood}, the formula in \Cref{thm:main-result} is also true for $k=1$, since in this case $\complexTuple{z} = \left(z_1\right)$ and $$\SpehSpherical{\complexTuple{z}}{c}\left(\uniformizer^{\complexTuple{\mu}, c}\right)= \abs{\det \uniformizer^{\complexTuple{\mu}, c}}^{z_1} = q^{-\sizeof{\complexTuple{\mu}} z_1} = \htHallLittlewood_{\complexTuple{\mu}}\left(q^{-z_1}; q\right).$$
\end{remark}

The following theorem summarizes the results from \Cref{subsec:special-values-of-spherical-k-c-function} and \Cref{thm:main-result}.
\begin{theorem}\label{thm:special-values-theorem-summary}
	Let $\complexTuple{z} =\left(z_1,\dots,z_k\right)$ be a tuple of complex numbers with $z_i - z_j \ne \pm 1,\pm 2,\dots,\pm c  \left(\bmod \frac{2 \pi i}{\log q} \zIntegers \right)$ for every $i, j$. Let $\tau = \principalSeries{\complexTuple{z}}$ be the principal series representation $\Ind{B_k}{\GL_k\left(\localField\right)}{\abs{\cdot}^{z_1} \boxtimes \dots \boxtimes \abs{\cdot}^{z_k}}$ and let $\SpehSpherical{\complexTuple{z}}{c} \in \Whittaker\left( \SpehRepresentation{\tau}{c}, \fieldCharacterkc{k}{c}\right)$ be the unique spherical element satisfying $\SpehSpherical{\complexTuple{z}}{c}\left(\IdentityMatrix{kc}\right) = 1$. Then
	\begin{enumerate}
		\item The assignment $\GL_c\left(\localField\right) \to \cComplex$ given by $$g \mapsto \SpehSpherical{\complexTuple{z}}{c}\left(\diag\left(g, \IdentityMatrix{\left(k-1\right)c}\right)\right)$$ is $\GL_c\left(\ringOfIntegers\right)$-bi-invariant.
		\item For any tuple $\complexTuple{\mu} = \left(\mu_1,\dots,\mu_{\ell}\right)$ of length $\ell \le c$ of weakly decreasing integers, let $\uniformizer^{\complexTuple{\mu}, kc} \coloneqq \diag\left(\uniformizer^{\mu_1},\dots,\uniformizer^{\mu_\ell}, \IdentityMatrix{kc - \ell}\right)$. If $\mu_{\ell} < 0$, then $\SpehSpherical{\complexTuple{z}}{c}\left(\uniformizer^{\complexTuple{\mu}, kc}\right) = 0$.
		\item For any partition $\mu \vdash m$ of length $\ell \le c$, $$\SpehSpherical{\complexTuple{z}}{c}\left(\uniformizer^{\complexTuple{\mu}, kc}\right) = \delta_{\ParabolicForWss{k}{c}}^{\frac{1}{2}}\left(\uniformizer^{\complexTuple{\mu}, kc}\right) \cdot \htHallLittlewood_{\complexTuple{\mu}}\left(q^{-z_1},\dots,q^{-z_k};q\right),$$
		where $\htHallLittlewood_{\complexTuple{\mu}}\left(x_1,\dots,x_k; q\right)$ is the modified Hall--Littlewood polynomial (see \Cref{sec:hall-littlewood-polynomials}).
	\end{enumerate}
\end{theorem}

As a corollary of \Cref{thm:main-result} and of the Shintani formula (\Cref{thm:shintani-formula}) combined with the formula $\delta_{B_k}^{\frac{1}{2}}\left(
			\uniformizer^{\complexTuple{\lambda}, k}\right) = q^{-\frac{\sizeof{\complexTuple{\lambda}}\left(k-1\right)}{2} + \nOfPartition\left(\lambda\right)}$, we get the following formula that expresses values of $\rSpehSpherical{\complexTuple{z}}{c}$ as linear combinations of values of $\WhittakerSpherical{\complexTuple{z}}$. It can be seen as a generalization of \cite[Theorem 5]{Sate2005}.
\begin{corollary}\label{cor:expansion-of-diag-in-terms-of-schur-functions}
	For any partition $\complexTuple{\mu} \vdash m$ of length $\ell \le c$, we have
	$$\rSpehSpherical{\complexTuple{z}}{c}\left(\uniformizer^{\complexTuple{\mu, c}}\right) = q^{-\frac{m\left(c-1\right)(k-1)}{2}} \sum_{\complexTuple{\lambda} \vdash m} q^{-\nOfPartition\left(\complexTuple{\lambda}\right)} \cdot \coKostkaFoulkes_{\complexTuple{\lambda}, \complexTuple{\mu}}\left(q\right) \cdot \WhittakerSpherical{\complexTuple{z}}\left(\uniformizer^{\complexTuple{\lambda}, k}\right),$$
	where the sum on the right hand side is over all partitions $\complexTuple{\lambda}$ of length $\ell\left(\complexTuple{\lambda}\right) \le k$ and $\coKostkaFoulkes_{\complexTuple{\lambda},\complexTuple{\mu}}\left(q\right) = q^{\nOfPartition\left(\complexTuple{\mu}\right)} \kostkaFoulkes_{\complexTuple{\lambda},\complexTuple{\mu}}\left(q^{-1}\right)$ is the cocharge Kostka--Foulkes polynomial (see \cite[Definition 3.4.13]{Haiman2003}).
\end{corollary}
\begin{remark}\label{rem:unipotent-character-at-unipotent-conjugacy-class}
	By \cite[Proposition 3.4.14]{Haiman2003}, the value $\coKostkaFoulkes_{\complexTuple{\lambda},\complexTuple{\mu}}\left(q\right)$ is the value of the character of the unipotent representation of $\GL_m\left(\finiteField_q\right)$ corresponding to the partition $\lambda$, evaluated at the unipotent conjugacy class $J_{\complexTuple{\mu}}\left(1\right)$, under the convention that the partition $(m)$ corresponds to the trivial representation (as in \cite{Green55, macdonald1998symmetric}). In more details, an irreducible representation of $\GL_m\left(\finiteField_q\right)$ is called unipotent if it can be realized as a subrepresentation of the parabolically induced representation $\Ind{B_m\left(\finiteField_q\right)}{\GL_m\left(\finiteField_q\right)}{1}$. The equivalence classes of irreducible subrepresentations of the aforementioned parabolic induction are in bijection with irreducible representations of the symmetric group $\SymmetricGroup_m$, which in turn are in bijection with partitions of $m$.
\end{remark}
\begin{remark}
	When $c=1$, and thus $\complexTuple{\mu} = \left(m\right)$ for some $m \ge 0$, \Cref{cor:expansion-of-diag-in-terms-of-schur-functions} recovers the definition, namely $$\SpehSpherical{\complexTuple{z}}{c}\begin{pmatrix}
		\uniformizer^{m}\\
		& \IdentityMatrix{k - 1}
	\end{pmatrix} =  \WhittakerSpherical{\complexTuple{z}}\begin{pmatrix}
		\uniformizer^{m}\\
		& \IdentityMatrix{k - 1}
	\end{pmatrix}.$$ This is because $\coKostkaFoulkes_{\complexTuple{\lambda},\complexTuple{\mu}}\left(q\right)$ is zero unless $\complexTuple{\lambda}$ dominates $\complexTuple{\mu}$ in the dominance partial order. Hence, in the case we have in hand $\coKostkaFoulkes_{\complexTuple{\lambda},\complexTuple{\mu}}\left(q\right)$ is zero unless $\complexTuple{\lambda} = \complexTuple{\mu} = \left(m\right)$, in which case $\coKostkaFoulkes_{\complexTuple{\lambda},\complexTuple{\mu}}\left(q\right) = 1$ (as it is the value of the character of the trivial representation at $J_{(m)}\left(1\right)$).
\end{remark}
\begin{remark}\label{rem:unramified-computation-of-ginzburg-kaplan-integral-using-modified-hall-littlewood}
	One can use \Cref{thm:main-result} to prove \Cref{thm:ginzburg-unramified-computation} (hence, these theorems are equivalent). To do so, start with \eqref{eq:evaluation-of-unramified-data} and replace $\rSpehSpherical{\complexTuple{z}}{c}\left(\uniformizer^{\complexTuple{\lambda}, c}\right)$ with the expression from \Cref{thm:main-result} (with $\complexTuple{\lambda} = \complexTuple{\mu}$) and use \eqref{eq:modified-hall-littlewood-definition} and then \eqref{eq:expansion-of-transformed-hall-littlewood-in-schur-basis}. We arrive at the following expression for $\zetaIntegral\left(s, \SphericalSection{\complexTuple{z}'}, \SphericalSection{-\complexTuple{z}'}, \SpehSpherical{\complexTuple{z}}{c} ; \fieldCharacter \right)$:
	\begin{equation*}
	\sum_{m = 0}^{\infty} T^{m}  \sum_{\complexTuple{\lambda}, \complexTuple{\lambda}' \vdash m} \pHallLittlewood_{\complexTuple{\lambda}}\left(\complexTuple{x}; q^{-1}\right) \kostkaFoulkes_{\complexTuple{\lambda}', \complexTuple{\lambda}}\left(q^{-1}\right)  s_{\complexTuple{\lambda}'}\left(\complexTuple{y}\right).
\end{equation*}	
	By \cite[Definition 3.4.2]{Haiman2003}, for any $\lambda' \vdash m$, $$\sum_{\complexTuple{\lambda} \vdash m} \pHallLittlewood_{\complexTuple{\lambda}}\left(\complexTuple{x}; q^{-1}\right) \kostkaFoulkes_{\complexTuple{\lambda}', \complexTuple{\lambda}}\left(q^{-1}\right) = s_{\complexTuple{\lambda}'}\left(\complexTuple{x}\right).$$
	Thus $$\zetaIntegral\left(s, \SphericalSection{\complexTuple{z}'}, \SphericalSection{-\complexTuple{z}'}, \SpehSpherical{\complexTuple{z}}{c} ; \fieldCharacter \right) = \sum_{m = 0}^{\infty} \sum_{\complexTuple{\lambda}' \vdash m} s_{\complexTuple{\lambda}'}\left(\complexTuple{x}\right) s_{\complexTuple{\lambda}'}\left(\complexTuple{y}\right) T^m.$$
	\Cref{thm:ginzburg-unramified-computation} now follows from the Cauchy identity for Schur polynomials \eqref{eq:cauchy-identity-for-schur-polynomials}.
\end{remark}

\section*{Acknowledgments}
In June 2024, I attended the conference ``Solvable Lattice Models, Number Theory and Combinatorics'' at Trinity College, Dublin, supported by the Hamilton Mathematics Institute at the School of Maths, Trinity College Dublin, by Boston College and by the NSF (Grant 2401464). The conference inspired me to write this work. I was especially was inspired by the talk given by Valentin Buciumas. I would like to thank the organizers for a wonderful event, and would especially like to thank Solomon Friedberg for the invitation to participate and Daniel Bump, Fan Gao, Nadya Gurevich and Edmund Karasiewicz for engaging conversations.

I would like thank Eyal Kaplan for his encouragement and for updating me about the current status of this question. I would also like to thank Yiannis Sakellaridis, Erez Lapid and Zhengyu Mao for their feedback and encouragement.

Finally, I would like to thank the anonymous referees for their comments that improved this paper.

\bibliographystyle{abbrv}
\bibliography{references}
\end{document}